\documentclass[11pt,reqno]{amsproc}

\title[]{Interior Electroneutrality in Nernst-Planck-Navier-Stokes Systems}

\author{Peter Constantin}
\address{Department of Mathematics, Princeton University, Princeton, NJ 08544}
\email{const@math.princeton.edu}

\author{Mihaela Ignatova}
\address{Department of Mathematics, Temple University, Philadelphia, PA 19122}
\email{ignatova@temple.edu}

\author{Fizay-Noah Lee}
\address{Program in Applied and Computational Mathematics, Princeton University, Princeton, NJ 08544}
\email{fizaynoah@princeton.edu}

\usepackage[margin=1in]{geometry}
\usepackage{amsmath, amsthm, amssymb}
\usepackage{times}
\usepackage{color}
\usepackage{hyperref}
\usepackage{comment}

\newcommand{\pa}{\partial}
\newcommand{\la}{\label}
\newcommand{\fr}{\frac}
\newcommand{\na}{\nabla}
\newcommand{\be}{\begin{equation}}
\newcommand{\ee}{\end{equation}}
\newcommand{\bal}{\begin{aligned}}
\newcommand{\eal}{\end{aligned}}
\newcommand{\ba}{\begin{array}{l}}
\newcommand{\ea}{\end{array}}

\newtheorem{thm}{Theorem}

\newtheorem{rem}{Remark}

\newtheorem{cor}{Corollary}

\newcommand{\beg}{\begin}

\renewcommand{\div}{{\mbox{div}\,}}
\newcommand{\D}{\Delta}

\date{today}
\begin{document}
\begin{abstract}
We consider the limit of vanishing Debye length for ionic diffusion in fluids, described by the Nernst-Planck-Navier-Stokes system. In the asymptotically stable cases of blocking (vanishing normal flux) and uniform selective (special Dirichlet) boundary conditions for the ionic concentrations, we prove that the ionic charge density $\rho$ converges in time to zero in the interior of the domain, in the limit of vanishing Debye length  ($\epsilon\to 0$). For the unstable regime of Dirichlet boundary conditions for the ionic concentrations, we prove bounds that are uniform in time and $\epsilon$. We also consider electroneutral boundary conditions, for which we prove that electroneutrality $\rho\to 0$ is achieved at any fixed $\epsilon> 0$, exponentially fast in time in $L^p$, for all $1\le p<\infty$. The results hold for two oppositely charged ionic species with arbitrary ionic diffusivities, in bounded domains with smooth boundaries. 
\end{abstract}

\noindent

\maketitle
\begin{section}{Introduction}
Interior electroneutrality is the vanishing of electrical charge away from boundaries.  This is an equilibrium feature of electrolytes in fluids, at distances larger than the Debye length from charged boundaries.  Ionic diffusion of electrolytes in solvents is decribed by the  Nernst-Planck-Navier-Stokes (NPNS) system. We study the NPNS system in an open connected  bounded domain $\Omega\subset\mathbb{R}^d$, $d=2,3$ with smooth boundary. The domain need not be simply connected. The system describes the time evolution of ionic concentrations in a Newtonian fluid \cite{rubibook}. Ions are transported by the fluid, and diffuse under the influence of the gradient of an electrochemical potential generated by the local charge density $\rho$ and the applied voltage on the boundaries. The fluid is forced by the electrical force exerted by the ionic charges. The system is given by the Nernst-Planck equations
\be
\pa_t c_i +u\cdot\na c_i =D_i\div(\na c_i+z_ic_i\na\Phi)
\la{np}
\ee
for $i=1,2,...,m,$ coupled to the Poisson equation
\be
-\epsilon\D\Phi=\rho
\label{poisson}
\ee
and to the Navier-Stokes equations
\be
\pa_t u+u\cdot\na u-\nu\D u+\na p=-K\rho\na\Phi,\quad \na\cdot u=0
\ee
where
\be
\rho=\sum_{i=1}^mz_ic_i.
\ee
The function $c_i(x,t)$ represents the local concentration of the $i$-th ionic species, $\Phi(x,t)$ is the electrical potential and $\rho(x,t)$ is the local charge density. The constants $z_i$ and $D_i > 0$ are, respectively, the ionic valence and ionic diffusivity of  the $i$-th ionic species. Although it is sometimes mathematically inconvenient, for most applications it is important to allow for unequal diffusivities. The constant $\epsilon>0$ is a rescaled dielectric permittivity of the solvent and is proportional to the square of the Debye length. The Debye length is typically very small, of the order of a few nanometers in electrolytes. The kinematic viscosity of the fluid is given by $\nu>0$, and $K$ is a positive coupling constant given by the product of the Boltzmann constant $k_B$ and the absolute temperature of the system $T_K$. The electrical potential $\Phi$ and charge density $\rho$ have been nondimensionalized so that $(k_BT_K/e)\Phi$ and $e\rho$ respectively are their dimensional counterparts, where $e$ is elementary charge.

The electrochemical potentials are
\be
\mu_i = \log{c_i} + z_i\Phi
\la{mui}
\ee
for $i=1,2,...,m$. In terms of the electrochemical potential, the Nernst-Planck equations (\ref{np}) are given by
\be
(\pa_t  +  u\cdot\na)c_i = D_i\div(c_i\na\mu_i), \quad  i=1,\dots m.
\la{npmu}
\ee

In this work, we consider the case of two ionic species, $m=2$, with ionic valences $z_1=1$ and $z_2=-1$. The boundary conditions for $\Phi$ are inhomogeneous Dirichlet boundary conditions
\be
\Phi(x,t)_{|\pa\Omega}=W(x)
\la{W}
\ee
where $W(x)$ is a given function of space, which we assume to be time independent and smooth. The boundary conditions for the Navier-Stokes equations are no slip,  homogeneous Dirichlet,
\be
u_{|\pa\Omega}=0.
\ee
We consider four sets of boundary conditions for $c_i$. Blocking boundary conditions, which correspond to impermeable boundaries that block ionic transport are
\be
(\textbf{BL})\quad n\cdot(\na c_i+z_ic_i\na\Phi)_{|\pa\Omega}=0,\quad i=1,2
\ee
where $n$ is the outward pointing unit normal vector along $\pa\Omega$.   
Therefore, the blocking boundary conditions are homogeneous Neumann conditions for the electrochemical potentials,
\be
(\textbf{BL})\quad  n\cdot \na{\mu_i}_{|\pa\Omega}=0,\quad i= 1,2.
\la{mublock}
\ee

Dirichlet boundary conditions for the ionic concentrations model ion-selective (or permselective) membranes along which a fixed concentration of ions is maintained. They are
\be
(\textbf{DI})\quad {c_i}_{|\pa\Omega}=\gamma_i,\quad i=1,2,
\ee
where $\gamma_i=\gamma_i(x)$ are positive smooth time-independent functions on the boundary. In view of the Dirichlet boundary condition (\ref{W}) for the potential, the Dirichlet boundary conditions are inhomogeneous Dirichlet boundary conditions for the electrochemical potentials
\be
(\textbf{DI})\quad {\mu_i}_{|\pa\Omega} = \log(\gamma_i) + z_i W,\quad i=1,2.
\ee

Uniform selective boundary conditions \cite{ci, np3d} are
\be
(\textbf{US}) \quad {c_i}_{|S_i}=\gamma_i,\, n\cdot(\na c_i+z_ic_i\na\Phi)_{|\pa\Omega\backslash S_i}=0,\quad i=1,2
\ee
where $S_i\subset\pa\Omega$ are boundary portions. We require additionally that
\be
(\log\gamma_i(x)+z_i W(x))_{|S_i}=\log Z_i^{-1}, 
\la{us}
\ee
holds with $Z_i>0$ {\em{constant in space and time}} when $S_i\neq \emptyset$. We take this to hold for at least one of $i=1,2$, otherwise the boundary conditions ({\textbf{US}}) coincide with ({\textbf{BL}}). 
The uniform selective boundary conditions thus require the constancy of the electrochemical potential on a portion of the boundary, and the vanishing of its normal derivative on the rest of the boundary,   
\be
(\textbf{US})\quad {\mu_i}_{|S_i} = \log Z_i^{-1}, \quad\quad \quad  n\cdot \na{\mu_i}_{|\pa\Omega\setminus S_i} = 0.
\la{usmui}
\ee

Electroneutral boundary conditions are
\be
(\textbf{EN})\quad {c_1}_{|\pa\Omega}={c_2}_{|\pa\Omega},\, n\cdot\na(c_1+c_2)_{|\pa\Omega}=0.
\ee
We denote the total salt concentration by $\sigma= c_1+c_2$. In terms of $\rho$ and $\sigma$, the electroneutral boundary conditions (\textbf{EN}) are homogeneous Dirichlet boundary conditions for the charge density and homogeneous Neumann conditions for the salt concentration, 
\be
\rho_{|\pa\Omega}=0,\,n\cdot\na\sigma_{|\pa\Omega}=0.
\la{EN}
\ee

In the absence of requirement (\ref{us}), the  boundary conditions (\textbf{DI})  are an example of general selective boundary conditions \cite{ci}. The choice of boundary conditions for the ionic concentrations and  the electrical potential plays an important role in the dynamics of the solutions of the NPNS system, and several different boundary conditions, including those considered in this paper, have been studied in the literature. For two dimensions and (\textbf{BL}) boundary conditions, global well posedness and asymptotic behavior are obtained in \cite{biler,biler2,choi,gajewski} for the system without fluid. The full  NPNS system with (\textbf{BL}) in two dimensions is addressed in \cite{bothe,ryham}, where for the electrical potential $\Phi$, Robin and homogeneous Dirichlet boundary conditions, respectively, are taken and global existence and stability are shown. For three dimensions, global well posedness is known in some but not all physically relevant cases. The lack of well posedness results in three dimensions is not solely due to the coupling with the Navier-Stokes equations, for which global existence is a major open problem. Even for the system uncoupled to fluid flow or the system coupled to Stokes flow instead, global well posedness in three dimensions is unknown in full generality. In \cite{schmuck}, global existence of weak solutions in three dimensions is shown for homogeneous Neumann boundary conditions on the potential. Recently, in \cite{liu}, the authors obtained analogous results in the case of no boundaries, $\Omega=\mathbb{R}^3$.

For (\textbf{BL}) and (\textbf{US}) boundary conditions for the ionic concentrations and inhomogeneous Dirichlet boundary conditions for the potential, global existence of strong solutions of the NPNS system is known in two dimensions for arbitrary large initial data \cite{ci}. For the same boundary conditions in three dimensions, global smooth solutions exist for initial conditions that are sufficiently small perturbations of steady state solutions \cite{np3d}. In all these cases with global existence, as time tends to infinity, the solutions converge to unique stationary solutions selected by the initial total concentrations and the boundary conditions. In the Dirichlet case (\textbf{DI}), global existence was shown in \cite{cil} for any spatial dimension by establishing uniform bounds depending on the parameter $\epsilon$. In this latter case, we do not expect stability results, as numerical simulations, experiments and rigorous analysis of simplified models suggest that instabilities may occur in this regime \cite{davidson,rubinstein,rubizaltz,zaltzrubi}.

The main results of this paper are as follows. In the cases with global existence and stability (\textbf{BL}) and (\textbf{US}), we show that the charge density $\rho$ vanishes in the interior of the domain $\Omega$ in the long time limit $t\to\infty$, in the limit of vanishing Debye length, $\epsilon\to 0$. That is, for any fixed initial conditions in 2D and any compact $K$ included in $\Omega$, we have
\be
\lim_{\epsilon\to 0}\lim_{t\to\infty}\sup_{x\in K}|\rho(x,t)| = 0.
\la{rhotozero}
\ee
The same result holds in 3D with the same boundary conditions, for small perturbations of steady states. This  result is a mathematical verification of the physical fact that in the stable cases, electroneutrality ($\rho\sim 0$) holds away from the boundaries. 

In \cite{ci} it was shown that for each fixed $\epsilon>0$ the solutions $(c_1,c_2,\Phi,u)$ converge in time to the steady state solutions 
$({c_1^*}_{\epsilon}, {c_2^*}_{\epsilon},\Phi^*_{\epsilon},0)$, where $\Phi^*_{\epsilon}$ is the unique solution of the Poisson-Boltzmann equation
\be
(\textbf{PB}_\epsilon)\quad -\epsilon\D\Phi^*_{\epsilon}=\fr{e^{-\Phi^*_{\epsilon}}}{Z_1}-\fr{e^{\Phi^*_{\epsilon}}}{Z_2}
\la{pb}
\ee
subject to Dirichlet data (\ref{W}). The constants 
$Z_i$, $i=1,2$,  are given by  
\be
Z_i^{-1}=I_0\left(\int_\Omega e^{-z_i\Phi^*_{\epsilon}}\,dx\right)^{-1}
\la{zibl}
\ee
with $z_1=1,\,z_2=-1$ in the (\textbf{BL}) case. 
In the (\textbf{US}) case, $Z_i$ is given by (\ref{us}) if $S_i\neq \emptyset$. If $S_i=\emptyset$, then $Z_i$ is given by (\ref{zibl}), as in the (\textbf{BL}) case. 

With $\Phi^*_{\epsilon}$ thus defined, the stationary ionic concentrations ${c_i^*}_{\epsilon}$ are given by the Boltzmann states,
\be
{c_i^*}_{\epsilon}(x)=\fr{e^{-z_i\Phi^*_{\epsilon}(x)}}{Z_i}.
\ee
The choices of $Z_i>0$ are precisely such that $\|c_i(0)\|_{L^1}=I_0=\|{c_i^*}_{\epsilon}\|_{L^1}$ in the (\textbf{BL}) case and such that ${c_i}_{|S_i}=\gamma_i={{c_i^*}_{\epsilon}}_{|S_i}$ in the (\textbf{US}) case.

In \cite{ci}, the convergence $c_i\rightarrow{c_i^*}_{\epsilon}$ as $t\to\infty$ holds in the space $H^1$. It is not difficult to verify that convergence in $L^\infty$ can be obtained by taking into consideration the uniform bounds in stronger norms (e.g. $L^2_t H^2_x$), also established in the same paper. Thus, we know that for each fixed $\epsilon>0$, $\|\rho(t)- \rho^*_\epsilon\|_{L^\infty}\to 0$ as $t\to \infty$  where $\rho^*_{\epsilon} = {c_1^*}_{\epsilon}- {c_2^*}_{\epsilon}$ is the charge density of the Boltzmann state. 
In order to prove the convergence (\ref{rhotozero}) in these cases,
what remains to be shown is that for each compact $K\subset\Omega$ the uniform convergence 
 \be
\lim_{\epsilon\to 0}\sup_{x\in K}\left|{\rho^*_\epsilon}(x)\right| =0
\la{rhostartozero}
\ee
holds. Here and in the rest of the paper, the subscript $\epsilon$ is used to emphasize that $\Phi^*_\epsilon$, and correspondingly $\rho^*_\epsilon$, arises as the solution to $(\textbf{PB}_\epsilon)$, for a specific choice of $\epsilon$ and boundary conditions.

In the case of the boundary conditions (\textbf{BL}) and (\textbf{US}) we distinguish three different types of (\textbf{US}) boundary conditions. One type is when  both the cation concentration ($c_1$) and anion concentration ($c_2$) have selective boundary portions (i.e. $S_i\neq\emptyset$ for both $i=1,2$), and two additional types are when one ionic species has a selective boundary portion while the other species is subject to purely blocking boundary conditions. Thus, in total, there are four different boundary conditions for which the uniform convergence (\ref{rhostartozero}) must be shown. These are proved in Theorems~\ref{pb2}--\ref{pb4} below. The proofs share some common elements and are based on the respective variational structures of the four Poisson-Boltzmann equations.

In the case of (\textbf{DI}) boundary conditions, we do not expect stability in general. We show that for arbitrary (\textbf{DI}) boundary conditions, the ionic concentrations do not grow larger than allowed by the Dirichlet and initial data. In particular, the ionic concentrations obey uniform bounds that do not depend on $\epsilon$ and consequently the charge densities are bounded uniformly, independently of the Debye length. The bound is obtained from a maximum principle for the two-by-two system of evolution equations for the concentrations. Such a uniform bound is not known in general.

 In the last section, we show that under (\textbf{EN}) boundary conditions, electroneutrality is achieved exponentially fast in $L^p$, $2\le p<\infty$. In this case
\be
\|\rho(t)\|_{L^p} \le C_pe^{-\lambda_pt}
\la{rholptozero}
\ee
holds from arbitrary initial data, at each fixed $\epsilon$,  with $C_p$, $\lambda_p$ independent of $\epsilon$. 
The a~priori upper bound (\ref{rholptozero}) is proved on the basis of the $p=2$ case and a maximum principle for the system with these boundary conditions.

\end{section}

\begin{section}{Asymptotic electroneutrality of equilibria}\la{eqel}

We consider here Poisson-Boltzmann equations corresponding to blocking and uniformly selective boundary conditions. In the first result, Theorem~\ref{pb2}, we address uniform selective boundary conditions in which both the anions and the cations have selective boundary conditions.

\begin{thm}\la{pb2}
Let  $Z_1,Z_2>0$ be fixed given positive constants, and let $\Phi_\epsilon^*$ be the unique solution of the Poisson-Boltzmann equation
\be
-\epsilon\D\Phi^*_\epsilon = \rho^*_\epsilon
\la{pbunif}
\ee
with 
\be
\rho^*_{\epsilon} = \fr{e^{-\Phi^*_\epsilon}}{Z_1}-\fr{e^{\Phi^*_\epsilon}}{Z_2},
\la{rhostareps}
\ee
and with boundary condition ${\Phi_\epsilon^*}_{|\partial\Omega}=W$. 
Then for each compact subset $K\subset\Omega$, we have
\be
\lim_{\epsilon\to 0}\sup_{x\in K}\left|\rho_\epsilon^*(x)\right| = 0.
\la{rhoepstozero}
\ee

\end{thm}

\begin{proof}

The proof of Theorem~\ref{pb2} uses the variational nature of the solution to the elliptic equation (\ref{pbunif}), (\ref{rhostareps}) with Dirichlet data $W$. The solution of this problem is the unique minimizer of the energy functional 
\be
J_\epsilon[\psi]=\int_\Omega \left(\fr{\epsilon}{2}|\na\psi|^2+\fr{e^{-\psi}}{Z_1}+\fr{e^{\psi}}{Z_2}\right)dx
\la{J}
\ee
on the set 
$\mathcal A=\{\psi\in H^1(\Omega)\,|\, e^{\psi},e^{-\psi}\in L^1(\Omega),\,\psi_{|\pa\Omega}=W\}$ \cite{ci}. We prove that the limit
\be
\lim_{\epsilon\to 0}\min_{\psi\in\mathcal{A}}J_\epsilon[\psi]=\lim_{\epsilon\to 0} J_\epsilon[\Phi_\epsilon^*]
\ee
exists and we compute it explicitly. This allows us to deduce the convergence of $\rho_\epsilon^*$ to $0$. We define
\be
G(y)=\fr{e^{-y}}{Z_1}+\fr{e^y}{Z_2}
\la{G}
\ee
so that $J_\epsilon[\psi]=\int_\Omega \left(\fr{\epsilon}{2}|\na\psi|^2+G(\psi)\right)dx$, and we also define
\be
Z=\fr{1}{2}\log\fr{Z_2}{Z_1}.
\la{Zdef}
\ee
We note that $G(y)$ attains its unique global minimum at $y=Z$.

\noindent{\bf Step 1.} \la{lim2} We have
\be
\lim_{\epsilon\to 0}\min_{\psi\in A}J_\epsilon[\psi] = \lim_{\epsilon\to 0}J_\epsilon[\Phi_\epsilon^*]=G(Z)|\Omega|.\la{limm2}
\ee

\noindent Indeed, the lower bound for $J_\epsilon[\cdot]$,
\be
J_\epsilon[\Phi_\epsilon^*]\ge\int_\Omega G(\Phi_\epsilon^*)\,dx\ge G(Z)|\Omega|,\la{lbb}
\ee
follows directly from (\ref{J}), (\ref{G}) and (\ref{Zdef}).
Next, we take as test functions $\psi_\delta\in \mathcal A$, which satisfy the properties 
1) $\psi_\delta-Z$ is supported in $\Omega\backslash \Omega_\delta$, 
with $\Omega_\delta=\{x\in\Omega\,|\, \inf_{y\in\pa\Omega}|x-y|>\delta\}$, 
2) $|\na\psi_\delta|\sim \mathcal{O}(\delta^{-1})$, 
and 
3) $|\psi_\delta|\le |Z|+ \sup |W|$. For such test functions, we see that
\be
J_\epsilon[\psi_\delta]\le \fr{\epsilon}{2}C_W\delta^{-1}+\int_\Omega \fr{e^{-\psi_\delta}}{Z_1}+\fr{e^{\psi_\delta}}{Z_2}\,dx
\ee
where $C_W$ is a constant depending on $W$ and $Z$ but is independent of $\delta$. Then choosing for instance $\delta(\epsilon) =\epsilon^{1/2}$ and applying the dominated convergence theorem for the second term on the right hand side, we obtain
\be
\limsup_{\epsilon\to 0}J_\epsilon[\Phi^*_\epsilon]\le\limsup_{\epsilon\to 0}J_\epsilon\left[\psi_{\delta(\epsilon)}\right]\le G(Z)|\Omega|.
\la{ubb}
\ee
Together with the lower bound (\ref{lbb}), the conclusion (\ref{limm2}) follows, and the proof of Step~1 is complete.

\noindent{\bf{Step 2.}} We claim that  
\be
\lim_{\epsilon\to 0}\Phi_\epsilon^*(x) = Z
\la{limphistarz}
\ee
holds uniformly for $x\in K$. 
To prove the claim, we first observe that because $G$ is convex, we have
\be
\D(G(\Phi_{\epsilon}^*))=G''(\Phi_{\epsilon}^*)|\na\Phi_{\epsilon}^*|^2+\fr{1}{\epsilon}G'(\Phi_{\epsilon}^*)^2\ge 0
\ee
where we used the fact that $\epsilon\D\Phi_{\epsilon}^*=G'(\Phi_{\epsilon}^*)$. Thus the function $x\mapsto G(\Phi_{\epsilon}^*(x))$ is subharmonic. Then, if $B=B_r$ is a ball centered at $x_0\in K$ with radius $r=d(\pa\Omega,\pa K)=\inf\{|x-y|\,|\, x\in\pa\Omega,\,y\in\pa K\}$, we have
\be
\int_BG(\Phi_{\epsilon}^*)\,dx\ge G(\Phi_{\epsilon}^*(x_0))|B|.
\ee
Thus
\be
\begin{aligned}
J_\epsilon[\Phi_\epsilon^*]\ge\int_\Omega G(\Phi_{\epsilon}^*)\,dx&=\int_{\Omega\backslash B} G(\Phi_{\epsilon}^*)\,dx+\int_B G(\Phi_{\epsilon}^*)\,dx\\
&\ge G(Z)|\Omega\backslash B|+G(\Phi_\epsilon^*(x_0))|B|\\
&= G(Z)|\Omega|+(G(\Phi_\epsilon^*(x_0))-G(Z))|B|.
\end{aligned}
\la{cont}
\ee
Thus, from (\ref{limm2}), we have 
\be
\lim_{\epsilon\to 0}J_{\epsilon}[\Phi_{\epsilon}^*]= G(Z)|\Omega|,
\la{conv}
\ee
and recalling that $G(y)$ attains its global minimum at $y=Z$, we obtain that $G(\Phi_\epsilon^*(x_0))\to G(Z)$ and thus that $\Phi_\epsilon^*(x_0)\to Z$ as $\epsilon\to 0$. The convergence is uniform in $K$ because we can choose a ball $B$ of radius $r$ for each $x_0\in K$ and because the convergence rate in (\ref{conv}) does not depend on the choice of $x_0$.   
 
The fact that (\ref{limphistarz}) holds completes the proof of Theorem~\ref{pb2} because we have that 
$\rho_\epsilon^*(x)=-G'(\Phi_\epsilon^*(x))$ and $G'(Z)=0$. \end{proof}

Next we consider the case of (\textbf{BL}) boundary conditions.
\begin{thm}\la{pb1} Let $I_0>0$ be given and let $\Phi_\epsilon^*$ be 
the unique solution of the Poisson-Boltzmann equation
\be
-\epsilon\D\Phi^*_\epsilon =\rho^*_\epsilon
\la{pbblock}
\ee
with
\be
\rho^*_{\epsilon} = I_0\left(\fr{e^{-\Phi^*_\epsilon}}{\int_\Omega e^{-\Phi^*_\epsilon}\,dx}-\fr{e^{\Phi^*_\epsilon}}{\int_\Omega e^{\Phi^*_\epsilon}\,dx}\right)
\la{rhostarepsi}
\ee
and with boundary condition ${\Phi_\epsilon^*}_{|\partial\Omega}=W$. Then for each compact $K\subset\Omega$, we have
\be
\lim_{\epsilon\to 0}\sup_{x\in K}\left |\rho_\epsilon^*(x)\right| = 0.
\la{limrhostarepsil}
\ee
\end{thm}
\begin{proof}
The unique solution of (\ref{pbblock})--(\ref{rhostarepsi}) with Dirichlet data $W$ is the minimizer of the functional
\be
I_\epsilon[\psi]=\fr{\epsilon}{2}\int_\Omega |\na\psi|^2\,dx+I_0\log\left(\int_\Omega e^{-\psi}\,dx\int_\Omega e^{\psi}\,dx\right)
\la{I}
\ee
on the set $\mathcal A=\{\psi\in H^1(\Omega)\,|\, e^{\psi},e^{-\psi}\in L^1(\Omega),\,\psi_{|\pa\Omega}=W\}$ \cite{ci}.

\noindent{\bf Step 1.} 
For the functional defined above in (\ref{I}), we have 
\be
\lim_{\epsilon\to 0}\min_{\psi\in A}I_\epsilon[\psi]=\lim_{\epsilon\to 0}I_\epsilon[\Phi_\epsilon^*]=2I_0\log|\Omega|. \la{limm1}
\ee

\noindent The proof of (\ref{limm1}) closely follows that of (\ref{limm2}). 
First, we observe that by Cauchy-Schwarz inequality we have
\be
\begin{aligned}
I_\epsilon[\Phi_\epsilon^*]&=\fr{\epsilon}{2}\int_\Omega |\na\Phi_\epsilon^*|^2\,dx+I_0\log\left(\int_\Omega e^{-\Phi_\epsilon^*}\,dx\int_\Omega e^{\Phi_\epsilon^*}\,dx\right)\\
&\ge I_0\log\left(\int_\Omega 1\,dx\right)^2\\
&=2I_0\log|\Omega|.
\la{lb}
\end{aligned}
\ee
Next, we take as test functions $\psi_\delta\in \mathcal A$ with the following properties: 1) $\psi_\delta$ is supported in $\Omega\backslash \Omega_\delta$, 2) $|\na\psi_\delta|\sim \mathcal{O}(\delta^{-1})$, and 3) $|\psi_\delta|\le \sup |W|$. Using these test functions, we obtain, as in the proof of (\ref{limm2}), taking for instance $\delta(\epsilon)=\epsilon^\fr{1}{2}$ and using dominated convergence,
\be
\limsup_{\epsilon\to 0} I_\epsilon[\Phi_\epsilon^*]\le\limsup_{\epsilon\to 0}I_\epsilon[\psi_{\delta(\epsilon)}]\le 2I_0\log|\Omega|.
\la{ub}
\ee
Combining (\ref{lb}) and (\ref{ub}), we obtain (\ref{limm1}).

\noindent{\bf Step 2.} We have
\be
\lim_{\epsilon\to 0}\int_\Omega e^{-\Phi_\epsilon^*}\,dx\int_\Omega e^{\Phi_\epsilon^*}\,dx=|\Omega|^2.
\la{limprod}
\ee

\noindent Indeed, we estimate as in (\ref{lb}),
\begin{align}
I_\epsilon[\Phi_\epsilon^*]\ge I_0\log\left(\int_\Omega e^{-\Phi_\epsilon^*}\,dx\int_\Omega e^{\Phi_\epsilon^*}\,dx\right)\ge 2I_0\log|\Omega|.
\end{align}
By (\ref{limm1}), the left hand side converges to $2I_0\log|\Omega|$ in the limit as $\epsilon\to 0$. Thus the middle term also converges to the same value, and (\ref{limprod}) follows.

\noindent{\bf Step 3.} We have the $L^1$ convergence  
\be
\lim_{\epsilon\to 0}\|\rho_\epsilon^*\|_{L^1} = 0.
\la{rhol1}
\ee

\noindent Toward the proof of (\ref{rhol1}), we set
\be
A_i(\epsilon)=\fr{1}{|\Omega|}\int_\Omega e^{-z_i\Phi_\epsilon^*}\,dx 
\ee
for $i=1,2$, with $z_1=1,\,z_2=-1$,  and we claim that
\be
\lim_{\epsilon\to 0}\left\|1-\fr{e^{-z_i\Phi_\epsilon^*}}{A_i(\epsilon)}\right\|_{L^1}=0.
\la{astozero}
\ee
In order to show (\ref{astozero}), we take advantage of the strong correlation between the two concentrations. Denoting by $\|\cdot\|$ the $L^2$ norm and by $(\cdot,\cdot)$ the $L^2$ inner product, we have
\be
\|u\|^2\|v\|^2=|(u,v)|^2+\|z\|^2\|v\|^2,
\la{cs}
\ee
where, assuming $|u|>0$ and $v\neq 0$ in $L^2$,
\be
z=u-\fr{(u,v)}{(v,v)}v=u\left(1-\fr{(u,v)}{(v,v)}\fr{v}{u}\right).
\ee
Setting $u=e^{{\Phi_\epsilon^*}/{2}}$ and $v=e^{-{\Phi_\epsilon^*}/{2}}$, (\ref{cs}) together with the Cauchy-Schwarz inequality gives 
\be
\begin{aligned}
\int_\Omega e^{\Phi_\epsilon^*}\,dx\int_\Omega e^{-\Phi_\epsilon^*}\,dx&=|\Omega|^2+\int_\Omega e^{\Phi_\epsilon^*}\left(1-\fr{e^{-\Phi_\epsilon^*}}{A_1(\epsilon)}\right)^2\,dx\int_\Omega e^{-\Phi_\epsilon^*}\,dx\\
&\ge |\Omega|^2 +\left(\int_\Omega \left|1-\fr{e^{-\Phi_\epsilon^*}}{A_1(\epsilon)}\right|\,dx\right)^2.
\end{aligned}
\ee
Then, since the left hand side converges to $|\Omega|^2$ by \eqref{limprod}, we obtain the conclusion (\ref{astozero}) for $i=1$. The $i=2$ case is obtained analogously by switching $u, v$ in the definition of $z$. 

From (\ref{astozero}) we have
\be
\lim_{\epsilon\to 0}\|\rho_\epsilon^*\|_{L^1}\le\fr{1}{|\Omega|}\lim_{\epsilon\to 0}\left(\left\|1-\fr{e^{-\Phi_\epsilon^*}}{A_1(\epsilon)}\right\|_{L^1}+\left\|1-\fr{e^{\Phi_\epsilon^*}}{A_2(\epsilon)}\right\|_{L^1}\right)=0,
\ee
and thus, (\ref{rhol1}) holds.

\noindent{\bf Step 4.} We prove bounds on $\Phi_\epsilon^*$ uniform in $\epsilon$. More precisely, for all $\epsilon>0$ and $x\in\Omega$, we have 
\be
\inf W\le \Phi_\epsilon^*(x)\le\sup W.
\la{max}
\ee

\noindent Indeed, suppose $\Phi_\epsilon^*$ attains an interior global maximum value exceeding $\sup W$, say at $x_0\in\Omega$. Then at $x_0$ we must have $\D\Phi_\epsilon^*\le 0$, so that
\be
\fr{e^{-\Phi_\epsilon^*(x_0)}}{\int_\Omega e^{-\Phi_\epsilon^*}{dx}} -\fr{e^{\Phi_\epsilon^*(x_0)}}{\int_\Omega e^{\Phi_\epsilon^*}{dx}}\ge  0
\la{ineq}
\ee
Rearranging the terms in \eqref{ineq}, we obtain
\be
\begin{aligned}
\left(e^{\Phi_\epsilon^*(x_0)}\right)^2 
&\le \fr{\int_\Omega e^{\Phi_\epsilon^*}{dx}}{\int_\Omega e^{-\Phi_\epsilon^*}{dx}}\\
&\le|\Omega |^{-2}\left(\int_\Omega e^{\Phi_\epsilon^*} dx\right)^2
\end{aligned}
\ee
where the second inequality follows from
\be
|\Omega|^2 \le\int_\Omega e^{\Phi_\epsilon^*}{dx}\int_\Omega e^{-\Phi_\epsilon^*}{dx}.
\la{cslb}
\ee
Thus, $\Phi^*_{\epsilon}$ must be constant, but this is a contradiction, because $\Phi_\epsilon^*(x_0)>\sup W$. The upper bound is proved and the lower bound is proved analogously. 

\noindent{\bf Step 5.} We have the $L^2$  convergence
\be
\lim_{\epsilon\to 0}\|\rho_\epsilon^*\|_{L^2}=0.
\la{l2toz}
\ee

\noindent Indeed, we obtain
\be
\lim_{\epsilon\to 0}\left\|1-\fr{e^{-z_i\Phi_\epsilon^*}}{A_i(\epsilon)}\right\|_{L^2}=0
\ee
for $i=1,2$, $z_1=1,\,z_2=-1$, from  (\ref{astozero}), because 
\be
\int_\Omega\left|1-\fr{e^{-z_i\Phi_\epsilon}}{A_i(\epsilon)}\right|^2\,dx\le \left\|1-\fr{e^{-z_i\Phi_\epsilon}}{A_i(\epsilon)}\right\|_{L^1}\left\|1-\fr{e^{-z_i\Phi_\epsilon}}{A_i(\epsilon)}\right\|_{L^\infty}
\ee
and, because  (\ref{max}) yields a uniform bound of the $L^\infty$ norm independent of $\epsilon$.

\noindent{\bf Step 6.} The map $x\mapsto (\rho_\epsilon^*(x))^2$ is subharmonic.

\noindent Recalling (\ref{rhostarepsi}), 
\be
\rho_\epsilon^*=I_0\left(\fr{e^{-\Phi^*_\epsilon}}{\int_\Omega e^{-\Phi^*_\epsilon}\,dx}-\fr{e^{\Phi^*_\epsilon}}{\int_\Omega e^{\Phi^*_\epsilon}\,dx}\right),
\la{rhoo}
\ee
a direct computation gives
\be
\D\rho_\epsilon^*=\rho_\epsilon^*\left(\fr{I_0}{\epsilon|\Omega|}\left(\fr{e^{-\Phi_\epsilon^*}}{A_1(\epsilon)}+\fr{e^{\Phi_\epsilon^*}}{A_2(\epsilon)}\right)+|\na\Phi_\epsilon^*|^2\right).
\ee
Thus, $\rho_\epsilon^*\D\rho_\epsilon^*\ge 0$, from which it follows that $\D(\rho_\epsilon^*)^2=2\rho_\epsilon^*\D\rho_\epsilon^*+2|\na\rho_\epsilon^*|^2\ge 0$.

We conclude now the proof of Theorem~\ref{pb1}.  As in the proof of Theorem~\ref{pb2}, we fix $x_0\in K$ and consider the ball $B=B_r$ centered at $x_0$ with radius $r=d(\pa\Omega,\pa K)$. By subharmonicity, we have
\be
\|\rho_\epsilon^*\|_{L^2}^2=\int_\Omega (\rho_{\epsilon}^*(x))^2\,dx\ge|B|(\rho_{\epsilon}^*(x_0))^2.
\la{conc}
\ee
Thus in light of (\ref{l2toz}), we obtain the desired conclusion (\ref{limrhostarepsil}).
\end{proof}

Finally, we consider the two remaining cases of (\textbf{US}) boundary conditions.
\begin{thm}\la{pb3}
Let $Z_1>0 $ and $I_2>0$ be given and let $\Phi_\epsilon^*$ be the unique solution of the Poisson-Boltzmann equation
\be
-\epsilon\D\Phi^*_\epsilon = \rho^*_\epsilon
\la{pbunif2}
\ee
with
\be
\rho^*_{\epsilon} = \fr{e^{-\Phi^*_\epsilon}}{Z_1}-I_2\fr{e^{\Phi^*_\epsilon}}{\int_\Omega e^{\Phi^*_\epsilon}\,dx}
\la{rhostarepsilo}
\ee
and with boundary condition ${\Phi_\epsilon^*}_{|\partial\Omega}=W$. Then for each compact $K\subset\Omega$, we have
\be
\lim_{\epsilon\to 0} \sup_{x\in K}\left|\rho_\epsilon^*(x)\right| = 0.
\la{limrhostarepsilon}
\ee
\end{thm}

\begin{proof} To prove Theorem~\ref{pb3}, it is useful to consider first the auxiliary problem of solving 
\be
-\epsilon\D \phi_\epsilon^*= \tilde{\rho}_\epsilon^*
\la{aux}
\ee
with
\be
\tilde{\rho}_{\epsilon} = \fr{e^{-\phi_\epsilon^*}}{\tilde Z_1}-I_2\fr{e^{\phi_\epsilon^*}}{\int_\Omega e^{\phi_\epsilon^*}\,dx}
\la{tilderho}
\ee
with boundary conditions
\be
\phi_\epsilon^*(x)_{|\pa\Omega}=\tilde{W}(x)=W(x)+w
\la{auxbc}
\ee
with $w>0$ a constant to be specified, and $\tilde Z_1$ defined by 
\be
\log {\tilde Z_1}^{-1}=\log\gamma_1+W+w=\log {Z_1}^{-1}+w.
\ee
We fix $w>0$ large enough so that 
\be
\fr{1}{\tilde Z_1}-\fr{I_2}{|\Omega|}>0.
\la{sign}
\ee
We emphasize that such a choice of $w$ depends on $\gamma_1$, $W$, $I_2$ and $|\Omega|$, but does not depend on $\epsilon$. The condition (\ref{sign}) is used below in the proof of (\ref{minusup}).

The unique solution of (\ref{aux})--(\ref{auxbc}) is the minimizer of the functional
\be
H_\epsilon[\psi]=\int_\Omega \left(\fr{\epsilon}{2}|\na\psi|^2+\fr{e^{-\psi}}{\tilde Z_1}\right)\,dx
+I_2\log\int_\Omega e^\psi\,dx
\la{H}
\ee
on the set $\tilde{\mathcal A}=\{\psi\in H^1(\Omega)\,|\, e^{\psi},e^{-\psi}\in L^1(\Omega),\,\psi_{|\pa\Omega}=\tilde W\}$.

For $H_\epsilon[\cdot]$, we observe that from Jensen's inequality, we have the lower bound
\be
H_\epsilon[\psi]\ge\int_\Omega \left(\fr{e^{-\psi}}{\tilde Z_1}+\fr{I_2}{|\Omega|}\psi\right)\,dx
+I_2\log|\Omega|
\la{jen}
\ee
so that defining
\be
K(y)=\fr{e^{-y}}{\tilde Z_1}+\fr{I_2}{|\Omega|}y,
\ee
we can write
\be
H_\epsilon[\psi]\ge\int_\Omega K(\psi)\,dx+I_2\log|\Omega|.
\la{jen'}
\ee
The function $K(y)$ attains its unique global minimum at $y=Z'$ where
\be
Z'=\log\fr{|\Omega|}{I_2\tilde Z_1}.
\la{zprime}
\ee

\noindent{\bf Step 1.} We have the limit
\be
\lim_{\epsilon\to 0}\min_{\psi\in\mathcal A}H_\epsilon[\psi]=\lim_{\epsilon\to 0}H_\epsilon[\phi_\epsilon^*]=K(Z')|\Omega|+I_2\log|\Omega|.
\la{Hlim}
\ee

\noindent We use a similar argument as in the proof of (\ref{limm2}). The lower bound corresponding to (\ref{lbb}) is given by
\be
H_\epsilon[\psi]\ge\int_\Omega K(\psi)\,dx+I_2\log|\Omega|\ge K(Z')|\Omega|+I_2\log|\Omega|.
\ee
The upper bound corresponding to (\ref{ubb}) is given by considering test functions $\psi_\delta\in\tilde{\mathcal{A}}$ satisfying 1) $\psi_\delta-Z'$ is supported in $\Omega\backslash \Omega_\delta$, 2) $|\na\psi_\delta|\sim \mathcal{O}(\delta^{-1})$, and 3) $|\psi_\delta|\le |Z'|+ \sup |\tilde W|$.

\noindent{\bf Step 2.} We have
\be
\limsup_{\epsilon\to 0}\int_\Omega e^{\phi^*_\epsilon}\,dx<\infty.
\la{intephup}
\ee
\noindent Indeed, by definition, we get
\be
H_\epsilon[\phi_\epsilon^*]
=\int_\Omega \left(\fr{\epsilon}{2}|\na\phi_\epsilon^*|^2+\fr{e^{-\phi_\epsilon^*}}{\tilde Z_1}\right)\,dx
+I_2\log\int_\Omega e^{\phi_\epsilon^*}\,dx\ge I_2\log\int_\Omega e^{\phi_\epsilon^*}\,dx
\ee
so the conlusion follows from (\ref{Hlim}).

\noindent{\bf Step 3.} We claim a uniform upper bound on $\phi_\epsilon^*$
\be
\limsup_{\epsilon\to 0}\sup_{x\in\Omega}\phi_{\epsilon}^*(x) <\infty.
\la{supphi}
\ee

\noindent The estimate
\be
e^{\phi_\epsilon^*}\le\max\left\{\left(\fr{1}{I_2\tilde Z_1}\int_\Omega e^{\phi_\epsilon^*} dx \right)^\fr{1}{2},\,e^{\sup \tilde W}\right\}
\la{ephup}
\ee
follows from a maximum principle argument. Indeed, if $\sup_\Omega\phi_\epsilon^*=\sup_{\pa\Omega}\phi_\epsilon^*$, then there is nothing to prove. So we may assume $\sup_\Omega\phi_\epsilon^*>\sup_{\pa\Omega}\phi_\epsilon^*$. Let $x_0\in\Omega$ be a point where the maximum value is attained. At this point we have $\D\phi^*_\epsilon\le 0$ and therefore $\tilde\rho_\epsilon^*\ge 0$. That is,
\be
\fr{e^{-\phi_\epsilon^*(x_0)}}{\tilde Z_1}-I_2\fr{e^{\phi_\epsilon^*(x_0)}}{\int_\Omega e^{\phi_\epsilon^*}\,dx}\ge 0,
\ee
thus, 
\be
\fr{1}{I_2\tilde Z_1}\int_\Omega e^{\phi_\epsilon^*}\,dx
\ge e^{2\phi_\epsilon^*(x_0)} 
\ge e^{2\phi_\epsilon^*},
\ee
which yields (\ref{ephup}). Now, from (\ref{ephup}) and (\ref{intephup}), we deduce (\ref{supphi}).

\noindent{\bf Step 4.} We have
\be
\limsup_{\epsilon\to 0}\int_\Omega e^{-\phi_\epsilon^*}\,dx<\infty.
\la{minusup}
\ee 

\noindent From (\ref{Hlim}) and (\ref{jen}), we obtain that there exists a constant $C$ independent of $\epsilon$ such that for all $\epsilon$ small enough,
\be
\int_\Omega \left(\fr{e^{-\phi_\epsilon^*}}{\tilde Z_1}+\fr{I_2}{|\Omega|}\phi_\epsilon^*\right)\,dx<C.
\ee
Then, because $|\phi_\epsilon^*|\le \exp|\phi_\epsilon^*|$, we have
\be
\int_\Omega \left(\fr{e^{-\phi_\epsilon^*}}{\tilde Z_1}-\fr{I_2}{|\Omega|}e^{|\phi_\epsilon^*|}\right)\,dx<C
\ee
from which it follows that
\be
\bal
\int_\Omega \fr{e^{-\phi_\epsilon^*}}{\tilde Z_1}\,dx-\int_{\phi_\epsilon^*<0}\fr{I_2}{|\Omega|}e^{-\phi_\epsilon^*}\,dx
<C+\int_{\phi_\epsilon^*\ge 0}\fr{I_2}{|\Omega|}e^{\phi_\epsilon^*}\,dx
<C'
\eal
\ee
where $C'$ is independent of $\epsilon$, due to (\ref{intephup}). 
The proof of (\ref{minusup}) is concluded by noting that the left hand side is bounded below by
\be
\left(\fr{1}{\tilde Z_1}-\fr{I_2}{|\Omega|}\right)\int_\Omega e^{-\phi_\epsilon^*}\,dx
\ee
and recalling (\ref{sign}).

\noindent{\bf Step 5.} We have
\be
\limsup_{\epsilon\to 0}\sup_{x\in \Omega}|\phi_\epsilon^*(x)|<\infty.
\la{phibound}
\ee

\noindent We claim that

\be
e^{-\phi_\epsilon^*}\le\max\left\{\left(\fr{I_2\tilde Z_1}{|\Omega|^2}\int_\Omega e^{-\phi_\epsilon^*}{dx} \right)^\fr{1}{2},\,e^{-\inf \tilde W}\right\}
\ee
holds. Indeed, we may ssume without loss of generality that $\inf_\Omega {\phi_\epsilon^*}<\inf_{\pa\Omega}\phi_\epsilon^*$. For $x_0\in\Omega$ such that $\phi_\epsilon^*(x_0)=\inf_\Omega \phi_\epsilon^*$, we know that $0\ge-\epsilon\D\phi_\epsilon^*(x_0)=\tilde\rho_\epsilon^*(x_0)$. That is,
\be
\fr{e^{-\phi_\epsilon^*(x_0)}}{\tilde Z_1}-I_2\fr{e^{\phi_\epsilon^*(x_0)}}{\int_\Omega e^{\phi_\epsilon^*}\,dx}\le 0
\ee
which, after rearranging and using Cauchy-Schwarz, gives
\be
e^{-2\phi_\epsilon^*}\le e^{-2\phi_\epsilon^*(x_0)}\le I_2\tilde Z_1\fr{1}{\int_\Omega e^{\phi_\epsilon^*}dx}\le\fr{I_2\tilde Z_1}{|\Omega|^2}\int_\Omega e^{-\phi_\epsilon^*} dx.
\ee
In view of (\ref{minusup})  we have that $e^{-\phi_\epsilon^*}$ (and hence $-\phi_\epsilon^*$) is uniformly bounded from above for all $\epsilon$ small enough. Together with (\ref{supphi}) we obtain (\ref{phibound}).

\noindent{\bf Step 6.} We claim
\be
\lim_{\epsilon\to 0}\left\|1-\fr{e^{\phi_\epsilon^*}}{\fr{1}{|\Omega|}\int_\Omega e^{\phi_\epsilon^*}dx}\right\|_{L^2}=0.
\la{l2conv}
\ee

\noindent We denote
\be
A_\epsilon= \fr{1}{|\Omega|}\int_\Omega e^{\phi_\epsilon^*}dx.
\ee
Doing a Taylor expansion of $\log x$ around $A_\epsilon$ and evaluating at $x=e^{\phi_\epsilon^*}$, we have
\be
\phi_\epsilon^*=\log e^{\phi_\epsilon^*}=\log A_\epsilon+\fr{1}{A_\epsilon}(e^{\phi_\epsilon^*}-A_\epsilon)-\fr{1}{2\xi^2}(e^{\phi_\epsilon^*}-A_\epsilon)^2
\ee
where $\xi>0$ is a value in between $e^{\phi_\epsilon^*}$ and $A_\epsilon$. In particular,
\be
\xi\le\max\{e^{\phi_\epsilon^*},A_\epsilon\}.
\ee
Thus, if view of the uniform bounds (\ref{phibound}) on $\phi_\epsilon^*$, we obtain that $\xi\le C$ for some $C>0$ independent of $\epsilon$, for all $\epsilon$ small enough. Therefore, for small $\epsilon$,
\be
\log A_\epsilon\ge \phi_\epsilon^*-\fr{1}{A_\epsilon}(e^{\phi_\epsilon^*}-A_\epsilon)+\fr{1}{2C^2}(e^{\phi_\epsilon^*}-A_\epsilon)^2.
\ee
Integrating and using the definition of $A_\epsilon$, we obtain
\be
\log A_\epsilon\ge\fr{1}{|\Omega|}\left(\int_\Omega \phi_\epsilon^* {dx} +\fr{1}{2C^2}\int_\Omega (e^{\phi_\epsilon^*}-A_\epsilon)^2{dx}\right).
\ee
Then we compute,
\be
\bal
H_\epsilon[\phi_\epsilon^*]\ge&\int_\Omega\fr{e^{-\phi_\epsilon^*}}{\tilde Z_1}\,dx+I_2\log\int_\Omega e^{\phi_\epsilon^*}\,dx\\
\ge&\int_\Omega\fr{e^{-\phi_\epsilon^*}}{\tilde Z_1}+\fr{I_2}{|\Omega|}\phi_\epsilon^*\,dx+\fr{I_2}{2C^2|\Omega|}\int_\Omega(e^{\phi_\epsilon^*}-A_\epsilon)^2{dx} + I_2\log|\Omega|\\
\ge&K(Z')|\Omega|+I_2\log|\Omega|+\fr{I_2}{2C^2|\Omega|}\int_\Omega (e^{\phi_\epsilon^*}-A_\epsilon)^2{dx}.
\eal
\ee
By (\ref{Hlim}), the left hand side converges to $K(Z')|\Omega|+I_2\log|\Omega|$ in the limit $\epsilon\to 0$. It follows that $e^{\phi_\epsilon^*}-A_\epsilon\to 0$ in $L^2$. Furthermore, due to the uniform bounds
(\ref{phibound}) on $\phi_\epsilon^*$, we obtain $1-{e^{\phi_\epsilon^*}}/{A_\epsilon}\to 0$ in $L^2$ in the limit $\epsilon\to 0$, thus completing the proof of (\ref{l2conv}).

\noindent{\bf Step 7.} We have
\be
\lim_{\epsilon\to 0}\left\|K(\phi_\epsilon^*)-K(Z')\right\|_{L^1} = 0.
\la{l1conv}
\ee

\noindent From the lower bound (\ref{jen'}) and the bound $K(\phi_\epsilon^*)\ge K(Z')$, we deduce that
\be
H_\epsilon[\phi_\epsilon^*]-(K(Z')|\Omega|+I_2\log|\Omega|)\ge\int_\Omega \left(K(\phi_\epsilon^*)-K(Z')\right)\,dx
\ee
and from (\ref{Hlim}), the left hand side converges to $0$ and thus (\ref{l1conv}) follows. 

\noindent{\bf Step 8.} We establish
\be
\lim_{\epsilon\to 0}\|\phi_\epsilon^*-Z'\|_{L^2}=0.
\la{phitoz}
\ee

\noindent As in the proof of (\ref{l2conv}), we consider the Taylor expansion of $K$,
\be
K(\phi_\epsilon^*)=K(Z')+K'(Z')(\phi_\epsilon^*-Z')+\fr{K''(\xi)}{2}(\phi_\epsilon^*-Z')^2
\ee
where $\xi$ is a value in between $Z'$ and $\phi_\epsilon^*$. In particular, since $K''(y)={e^{-y}}/{\tilde Z_1}$, we have 
\be
K''(\xi)\ge \min\left(\fr{e^{-\phi_\epsilon^*}}{\tilde Z_1},\fr{e^{-Z'}}{\tilde Z_1}\right).
\ee
Then, by the uniform bounds (\ref{phibound}) on $\phi_\epsilon^*$, we conclude that $K''(\xi)\ge C$ for some $C>0$ independent of $\epsilon$, for all $\epsilon$ small enough. Thus, recalling that $K'(Z')=0$, we obtain for small $\epsilon$,
\be
K(\phi_\epsilon^*)-K(Z')\ge\fr{C}{2}(\phi_\epsilon^*-Z')^2,
\la{kk}
\ee
and  (\ref{phitoz}) follows from (\ref{l1conv}) upon integrating (\ref{kk}) over $\Omega$ and taking the limit $\epsilon\to 0$.

\noindent{\bf Step 9.}  We have
\be
\lim_{\epsilon\to 0}\left\|\tilde\rho_\epsilon^*\right \|_{L^2}=0.
\la{tilderhotozero}
\ee

\noindent We recall from (\ref{zprime})
\be
\fr{I_2}{|\Omega|} =\fr{e^{-Z'}}{\tilde Z_1}
\la{expzprime}
\ee
and thus, in view of 
\begin{align*}
K(y)=\fr{e^{-y}}{\tilde Z_1}+\fr{I_2}{|\Omega|}y,
\end{align*}
we get 
\be
\fr{e^{-\phi^*_\epsilon}}{\tilde Z_1}-\fr{I_2}{|\Omega|}
=K(\phi^*_\epsilon)-K(Z')-\fr{I_2}{|\Omega|}(\phi_\epsilon^*-Z').
\la{triangle}
\ee
Therefore,
\be
\left\|\fr{e^{-\phi^*_\epsilon}}{\tilde Z_1}-\fr{I_2}{|\Omega|}\right\|_{L^1}\le \|K(\phi_\epsilon^*)-K(Z')\|_{L^1}+\fr{I_2}{|\Omega|}\|\phi_\epsilon^*-Z'\|_{L^1}.
\ee
Then (\ref{phitoz})  together with (\ref{l1conv}) gives 
\be
\lim_{\epsilon\to 0}\left\|\fr{e^{-\phi^*_\epsilon}}{\tilde Z_1}-\fr{I_2}{|\Omega|}\right\|_{L^1}=0.
\la{l11}
\ee
Next, since 
\be
\tilde\rho_\epsilon^*=\left(\fr{e^{-\phi_\epsilon^*}}{\tilde Z_1}-\fr{I_2}{|\Omega|}\right)+\left(\fr{I_2}{|\Omega|}-I_2\fr{e^{\phi_\epsilon^*}}{\int_\Omega e^{\phi_\epsilon^*}\,dx}\right),
\ee 
by the  triangle inequality, we obtain
\be
\bal
\|\tilde\rho_\epsilon^*\|_{L^2}&\le \left\|\fr{e^{-\phi^*_\epsilon}}{\tilde Z_1}-\fr{I_2}{|\Omega|}\right\|_{L^2}+\fr{I_2}{|\Omega|}\left\|1-\fr{e^{\phi_\epsilon^*}}{\fr{1}{|\Omega|}\int_\Omega e^{\phi_\epsilon^*}dx}\right\|_{L^2}\\
&\le  \left\|\fr{e^{-\phi^*_\epsilon}}{\tilde Z_1}-\fr{I_2}{|\Omega|}\right\|_{L^\infty}^\fr{1}{2}\left\|\fr{e^{-\phi^*_\epsilon}}{\tilde Z_1}-\fr{I_2}{|\Omega|}\right\|_{L^1}^\fr{1}{2}+\fr{I_2}{|\Omega|}\left\|1-\fr{e^{\phi_\epsilon^*}}{\fr{1}{|\Omega|}\int_\Omega e^{\phi_\epsilon^*}dx}\right\|_{L^2}.
\eal
\ee
Then the conclusion (\ref{tilderhotozero}) follows from the uniform bounds on $\phi_\epsilon^*$,  (\ref{l2conv}) and (\ref{l11}).

\noindent{\bf Step 10.} The map $x\mapsto (\tilde\rho_\epsilon^*(x))^2$ is subharmonic.

\noindent Indeed, from the definition (\ref{tilderho}) and the equation (\ref{aux}) it follows that
\be
\Delta\tilde\rho_{\epsilon}^* = \tilde\rho_{\epsilon}^*\left [|\na \phi^*_{\epsilon}|^2 + \fr{1}{\epsilon}\left(\fr{e^{-\phi_\epsilon^*}}{\tilde Z_1} + I_2\fr{e^{\phi_\epsilon^*}}{\int_\Omega e^{\phi_\epsilon^*}\,dx}
 \right)\right]
\la{deltarhotildestar}
\ee
and subharmonicity is deduced from $\Delta(\tilde\rho_{\epsilon}^*)^2 \ge 2\tilde\rho_{\epsilon}^* \Delta\tilde\rho_{\epsilon}^*$.

\noindent{\bf Step 11.} We claim that $\rho_\epsilon^*=\tilde\rho_\epsilon^*$. 

\noindent To prove the claim, we define $\Psi_\epsilon^*=\phi_\epsilon^*-w$. Then, we compute
\be
\bal
-\epsilon\D\Psi_\epsilon^*=-\epsilon\D\phi_\epsilon^*&=\fr{e^{-\phi_\epsilon^*}}{\tilde Z_1}-I_2\fr{e^{\phi_\epsilon^*}}{\int_\Omega e^{\phi_\epsilon^*}\,dx}\\
&=\fr{e^{-(\phi_\epsilon^*-w)}}{Z_1}-I_2\fr{e^{\phi_\epsilon^*-w}}{\int_\Omega e^{\phi_\epsilon^*-w}\,dx}\\
&=\fr{e^{-\Psi_\epsilon^*}}{ Z_1}-I_2\fr{e^{\Psi_\epsilon^*}}{\int_\Omega e^{\Psi_\epsilon^*}\,dx}
\la{bla}
\eal
\ee
where the second line follows from multiplying the first fraction by ${e^w}/{e^w}$ and the second by ${e^{-w}}/{e^{-w}}$. On the other hand, we have that $(\Psi_\epsilon^*)_{|\pa\Omega}=W$. Thus, by uniqueness of solutions of the Dirichlet problem (\ref{pbunif2}), we conclude that $\Psi_\epsilon^*=\Phi_\epsilon^*$, and the first line of (\ref{bla}) implies that $\rho_\epsilon^*=\tilde\rho_\epsilon^*$. 

The proof of Theorem~\ref{pb3} using steps 9, 10 and 11 is concluded following the same reasoning as in the proof of Theorem~\ref{pb1}, namely using (\ref{tilderhotozero}) and (\ref{conc}). 

\end{proof}

\begin{thm}\la{pb4}
Let $Z_2>0$ an $I_1>0$ be given, and let  $\Phi_\epsilon^*$ be the unique solution of the Poisson-Boltzmann equation
\be
-\epsilon\D\Phi^*_\epsilon = \rho^*_\epsilon
\la{pbunif3}
\ee
with
\be 
\rho^*_{\epsilon} = I_1\fr{e^{-\Phi^*_\epsilon}}{\int_\Omega e^{-\Phi_\epsilon^*}\,dx}-\fr{e^{\Phi^*_\epsilon}}{Z_2} 
\la{rhostarepsilon}
\ee
and with boundary condition ${\Phi_\epsilon^*}_{|\partial\Omega}=W$. Then for each compact $K\subset \Omega$, we have
\be
\lim_{\epsilon\to 0} \sup_{x\in K}\left|\rho_\epsilon^*(x)\right| = 0.
\la{limrhostarepsilonn}
\ee
\end{thm}
The proof of Theorem~\ref{pb4} is very similar to the proof of Theorem~\ref{pb3} and is omitted.
\end{section}

\begin{section}{Time Asymptotic Interior Electroneutrality}\la{taie}
In this section, we consider boundary conditions (\textbf{BL}) and (\textbf{US}) where global existence of smooth solutions and nonlinear stability of unique steady states have been established in \cite{ci} and \cite{np3d}. In these stable regimes, we show that $\rho\to 0$ uniformly on compact sets $K\subset\Omega$ in the limit of small $\epsilon$ and large time $t$.

\begin{thm} \la{blocktime} (Blocking Boundary Conditions) Let $(c_1,c_2,\Phi,u)$ be solutions of the 2D NPNS system subject to (\textbf{BL}) conditions. Assume that the initial conditions satisfy $c_i(0) \ge 0$, $c_i(0)\in H^1(\Omega)$, $i=1,2$,
$u(0)\in (H^1_0(\Omega))^2$, $\div u(0)=0$. We assume
 $\|c_1(0)\|_{L^1}=\|c_2(0)\|_{L^1}=I_0$. Then for any compact  $K\subset\Omega$ and $\delta>0$, there exists $\epsilon'=\epsilon'(K,\delta)>0$ such that for all $\epsilon\le\epsilon'$ there exists $t_{\epsilon}$ such that for all $t\ge t_\epsilon$, we have $\sup_{x\in K}|\rho(x,t)|\le\delta$. Thus,
\be
\lim_{\epsilon\to 0}\lim_{t\to\infty}\sup_{x\in K}|\rho(x,t)| =0.
\ee
\end{thm}
\begin{proof} Under the conditions of the theorem, from \cite{ci}, we have
that for each $\epsilon>0$ there exists a unique steady Boltzmann state solving (\ref{pbblock}), (\ref{rhostarepsi})
such that
\be
\lim_{t\to \infty}\|\rho(t) -\rho^*_{\epsilon}\|_{L^{\infty}} = 0.
\la{rhominusrho}
\ee
Then the result follows from Theorem \ref{pb1}.
\end{proof}

\begin{thm}\la{ustime} (Uniformly Selective Boundary Conditions) Let  $(c_1,c_2,\Phi,u)$ be solutions of the 2D NPNS system subject to (\textbf{US}) and corresponding to initial data lying in a compact subset of $H^1(\Omega)^2 \times H_0^1(\Omega)^2$  and obeying the natural side conditions $c_i(0)\ge 0$, $\div u(0) =0$. Then for any compact  $K\subset\Omega$ and $\delta>0$, there exists $\epsilon'=\epsilon'(K,\delta)>0$ such that for all $\epsilon\le\epsilon'$ there exists $t_{\epsilon}$ such that for all $t\ge t_\epsilon$, we have $\sup_{x\in K}|\rho(x,t)|\le\delta$.
Thus,
\be
\lim_{\epsilon\to 0}\lim_{t\to\infty}\sup_{x\in K}|\rho(x,t)| =0.
\ee
\la{unif}
\end{thm}
\begin{proof}
We distinguish three cases, depending on the type of boundary condition. If both $c_1$ and $c_2$ have selective boundary conditions, then there exists a unique Boltzmann state, with $Z_i$, $i=1,2$ fixed by (\ref{us}) and solving
(\ref{pbunif}) with $\rho^*$ given in (\ref{rhostareps}). By \cite{ci} the charge density of the solution of the NPNS system converges uniformly in time to $\rho^*_{\epsilon}$. Then the limit of vanishing Debye length $\epsilon\to 0$ follows from Theorem~\ref{pb2}.
If the cation concentration $c_1$ obeys selective boundary conditions and the anion concentration $c_2$ has blocking boundary conditions, that is if $S_1\neq \emptyset$ and $S_2 = \emptyset$, then the Boltzmann state is determined by constants $Z_1$ given in (\ref{us}) for $i=1$ and $I_2$ given by the integral of the initial data $I_2 = \int_{\Omega} c_2(0)dx$. The Boltzmann state obeys (\ref{pbunif2}) with $\rho^*_{\epsilon}$ given by (\ref{rhostarepsilo}).
By results in \cite{ci} it follows that for any $\epsilon>0$ the charge density $\rho$ of the NPNS solution converges uniformly in time to $\rho^*_{\epsilon}$.
The vanishing Debye length result follows then from Theorem~\ref{pb3}.
Finally, the case in which the anions have selective boundary conditions and the cations have blocking boundary conditions follows similarly, using~Theorem \ref{pb4}.
\end{proof}
 
\beg{rem} Solutions of NPNS in 3D with blocking or uniform selective boundary conditions and whose initial data are small perturbations of Boltzmann steady states, exist globally, are smooth, and converge in time to the Boltzmann steady state \cite{np3d}. The known basin of stability of the Boltzmann state depends on $\epsilon$. Theorems~\ref{pb2}--\ref{pb4} are valid in 3D and the proofs of
results corresponding to Theorems~\ref{blocktime}--\ref{ustime} are the same.
 \end{rem}

\end{section}

\begin{section}{Maximum Principle for Dirichlet Boundary Conditions}
We consider the case of Dirichlet boundary conditions, and prove bounds which are uniform in the Debye length.
\begin{thm}
We take a smooth solution $(c_1,c_2,\Phi,u)$ of the NPNS system with (\textbf{DI}) boundary conditions. We consider $d=2, 3$ and assume the initial data and boundary conditions are smooth. Then the ionic concentrations $c_i$ obey the following uniform in time bounds
\be
c_i(x,t)\le \Gamma=\max\left\{\sup_{\pa\Omega}\gamma_1,\sup_{\pa\Omega}\gamma_2,\sup_\Omega c_1(0),\sup_\Omega c_2(0)\right\},\quad i=1,2.
\ee
\la{mp}
\end{thm}
\begin{proof} We define $m_i(t)=\sup_\Omega c_i(x,t)$, $i=1,2$, and $M(t)=\max\{m_1(t),m_2(t)\}$. Fix $\Gamma'>\Gamma$. We suppose for the sake of contradiction that for some $t>0$, we have $M(t)\ge\Gamma'$. Then, by continuity, there exists a first time $t_0>0$ when $M(t_0)=\Gamma'$ is attained. Without loss of generality, we assume $m_1(t_0)=\Gamma'$. We distinguish two cases: $m_2(t_0)<\Gamma'$ and $m_2(t_0)=\Gamma'$. We first consider the case $m_2(t_0)<\Gamma'$. Since $\Gamma'>\sup_{\pa\Omega}\gamma_1$, there exists an interior point $x_0\in\Omega$ where $c_1(x_0,t_0)=\Gamma'$. Thus, evaluating (\ref{np}) at the maximal point $(x_0,t_0)$ and using (\ref{poisson}), we obtain
\be
\pa_t c_1(x_0,t_0)\le-\fr{D_1}{\epsilon}c_1(x_0,t_0)(c_1(x_0,t_0)-c_2(x_0,t_0))<-\fr{D_1}{\epsilon}\Gamma'(\Gamma'-\Gamma')=0
\la{<0}
\ee
where we used the fact that at an interior maximal point the gradient vanishes and the Laplacian is non-positive. This is a contradiction, since, by the choice of $t_0$, we have $\pa_t c_1(x_0,t_0)\ge 0$.

For the case $m_2(t_0)=\Gamma'$, we need a different argument. Since $m_1(t_0)=m_2(t_0)=\Gamma'>\Gamma$, by continuity we know that for a short time interval leading up to $t_0$, the maximal points for $c_i$, $i=1,2$, are attained in the interior. That is, there exists $\delta>0$ such that for all $s\in[t_0-\delta,t_0]$, there exist interior points $x_{i}(s)\in\Omega$ such that $c_i(x_{i}(s),s)=m_i(s)$. Thus for $i=1$ and for each $s\in[t_0-\delta,t_0]$ and $r<s$, we have
\be
\begin{aligned}
m_1(s)-m_1(r)&=\sup_\Omega c_1(x,s)-\sup_\Omega c_1(x,r)\\
&\le c_1(x_{1}(s),s)-c_1(x_{1}(s),r).
\end{aligned}
\ee
Then, dividing both sides by $s-r$ and taking the limit supremum, we obtain as in the case $m_2(t_0)<\Gamma'$,
\be
\begin{aligned}
\limsup_{r\to s^-}\fr{m_1(s)-m_1(r)}{s-r}&\le \pa_t c_1(x_{1}(s),s)\\
&\le -\fr{D_1}{\epsilon}m_1(s)(m_1(s)-m_2(s))\\
&\le -\fr{D_1}{2\epsilon}m_1^2(s)+\fr{D_1}{2\epsilon}m_2^2(s).
\end{aligned}
\la{m1}
\ee
Similarly for $i=2$, we deduce
\be
\begin{aligned}
\limsup_{r\to s^-}\fr{m_2(s)-m_2(r)}{s-r}\le  \fr{D_2}{2\epsilon}m_1^2(s)-\fr{D_2}{2\epsilon}m_2^2(s).
\end{aligned}
\la{m2}
\ee
Multiplying (\ref{m1}) by $D_2$ and (\ref{m2}) by $D_1$ and adding, we obtain
\be
\limsup_{r\to s^-} \left(D_2\fr{m_1(s)-m_1(r)}{s-r}+D_1\fr{m_2(s)-m_2(r)}{s-r}\right)\le 0
\ee
for any $s\in[t_0-\delta,t_0]$. 
Thus, by monotonicity we have
\be
D_2m_1(t_0-\delta)+D_1m_2(t_0-\delta)\ge D_2m_1(t_0)+D_1m_2(t_0)=(D_1+D_2)\Gamma'.
\ee
In other words, at a strictly earlier time than $t_0$, we have $m_1(t_0-\delta)\ge\Gamma'$ or $m_2(t_0-\delta)\ge\Gamma'$, in either case, a contradiction. Therefore the proof is complete.
\end{proof}
\beg{rem} In two dimensions the smoothness of solutions follows from the fact that the initial data and boundary conditions are smooth. The degree of smoothness required by the result is $c_i\in C^2(\Omega)\cap C^1(\bar{\Omega})$ locally in time. The velocity does not enter in a quantitative manner in the arguments. In the case of 3D NPNS, although it does not participate quantitatively, the velocity needs to be assumed to be regular enough for the Navier-Stokes solutions to be known to exist.
\end{rem}

\end{section}

\begin{section}{Electroneutral Boundary Conditions}
In this last section, we consider electroneutral boundary conditions (\textbf{EN}) and show that in this case, the charge density $\rho$ converges to $0$ at an exponential rate. In contrast to the results of Section~\ref{taie}, we can show here that the convergence holds for any fixed $\epsilon>0$, and the rate is independent of $\epsilon$. 

\begin{thm}\la{decay} For global solutions $(c_1,c_2,\Phi,u)$ of the NPNS system with (\textbf{EN}) boundary conditions, there exist constants $\lambda>0$, depending on $\Omega$ and $D_i$, and $C>0$, depending additionally on initial conditions, such that
\be
\|\rho(t)\|_{L^2}\le Ce^{-\lambda t},\quad t\ge 0
\la{eb}
\ee
holds.
\end{thm}
\begin{rem} Global regularity of solutions of NPNS in 2D, or of Nernst-Planck equations coupled to time dependent Stokes equations in 3D, is a consequence of the a~priori $L^2$ control (\ref{eb}) \cite{cil}. In the case of 3D NPNS, we must assume that the velocity is regular enough (for instance,  $u \in L^4(dt; H^1(\Omega)^3)$). The arguments in the proof of Theorem~\ref{decay} do not involve the velocity in a quantitative manner.
\end{rem}
\begin{proof}
We consider the equations satisfied by $\rho=c_1-c_2$ and $\sigma=c_1+c_2$. Dividing (\ref{np}) by $D_i$ and summing in $i$ we obtain
\be
D_t\left(\fr{c_1}{D_1}+\fr{c_2}{D_2}\right)=\D\sigma+\div(\rho\na\Phi)
\ee
where 
\be
D_t=\pa_t+u\cdot\na
\la{material}
\ee
is the material derivative. Similarly, if we subtract the equation for $i=2$ from that of $i=1$, we obtain
\be
D_t\left(\fr{c_1}{D_1}-\fr{c_2}{D_2}\right)=\D\rho+\div(\sigma\na\Phi).
\ee
Then, we observe
\be
\bal
\fr{c_1}{D_1}+\fr{c_2}{D_2}
=\fr{1}{D_1D_2}(D_2c_1+D_1c_2)&=\fr{1}{D_1D_2}((D_2-D_1)c_1+D_1(c_1+c_2))\\
&=\fr{1}{D_1D_2}\left(\fr{D_2-D_1}{2}c_1+\fr{D_1+D_2}{2}c_1+D_1c_2\right),
\eal
\ee
so, in term of $\rho$ and $\sigma$,
\be
\bal
\fr{c_1}{D_1}+\fr{c_2}{D_2}
&=\fr{1}{D_1D_2}\left(\fr{D_2-D_1}{2}\rho+\fr{D_1+D_2}{2}\sigma\right)\\
&=\fr{1}{D_1D_2}\left(\delta\rho+D\sigma\right)
\eal
\ee
where 
\be
\bal
\delta&=\fr{D_2-D_1}{2}\\
D&=\fr{D_1+D_2}{2}.
\la{deltaD}
\eal
\ee
Similar calculations give
\be
\fr{c_1}{D_1}-\fr{c_2}{D_2}=\fr{1}{D_1D_2}(D\rho+\delta\sigma).
\ee
Therefore, $\rho$ and $\sigma$ satisfy the differential equations
\begin{align}
\fr{1}{D_1D_2}D_t(\delta\rho+D\sigma)&=\D\sigma+\div(\rho\na\Phi)   \la{dD}\\
\fr{1}{D_1D_2}D_t(D\rho+\delta\sigma)&=\D\rho+\div(\sigma\na\Phi)   \la{Dd}.
\end{align}
Then, we multiply (\ref{dD}) by $\sigma$ and integrate by parts, and using the boundary conditions (\textbf{EN}), we get
\be
\fr{1}{D_1D_2}\left(\fr{d}{dt}\int_\Omega \fr{D}{2} \sigma^2\,dx+\int_\Omega \delta(\pa_t\rho)\sigma \,dx+\int_\Omega \delta(u\cdot\na\rho)\sigma\,dx\right)+\int_\Omega|\na\sigma|^2\,dx=-\int_\Omega \rho\na\Phi\cdot\na\sigma\,dx.
\la{sig}
\ee
Next, we multiply (\ref{Dd}) by $\rho$ and integrate by parts, and using the boundary conditions (\textbf{EN}) and the incompressibility condition $\div u=0$, we obtain
\be
\fr{1}{D_1D_2}\left(\fr{d}{dt}\int_\Omega \fr{D}{2} \rho^2\,dx+\int_\Omega\delta(\pa_t\sigma)\rho\,dx -\int_\Omega \delta(u\cdot\na\rho)\sigma\,dx\right)+\int_\Omega|\na\rho|^2\,dx=\int_\Omega \div(\sigma\na\Phi)\rho\,dx.
\la{rho}
\ee
For the integral on the right hand side, we use the Poisson equation for $\Phi$ to get
\be
\int_\Omega\div(\sigma\na\Phi)\rho\,dx=\int_\Omega\rho\na\Phi\cdot\na\sigma\,dx -\fr{1}{\epsilon}\int_\Omega\sigma\rho^2\,dx\le\int_\Omega\rho\na\Phi\cdot\na\sigma\,dx
\la{last}
\ee
where we used the fact that $\sigma\ge 0$, which in turn follows from the fact that $c_1,c_2\ge 0$. The initial concentrations are nonnegative $c_1(0),c_2(0)\ge 0$, and the nonnegativity in preserved  \cite{ci,cil}. Using (\ref{last}), we add (\ref{sig}) and (\ref{rho}) to deduce
\be
\fr{1}{D_1D_2}\fr{d}{dt}\int_\Omega \left(\fr{D}{2}(\rho^2+\sigma^2)+\delta\rho\sigma\right)\,dx
+\int_\Omega |\na\sigma|^2+|\na\rho|^2\,dx\le 0.
\la{l2diss'}
\ee
Next, we define
\be
\bar\sigma = \fr{1}{|\Omega|}\int_\Omega\sigma \,dx,\quad \bar\rho = \fr{1}{|\Omega|}\int_\Omega\rho \,dx.
\ee 
Then, we obtain 
\begin{align}
&\int_\Omega \left(\fr{D}{2}(\rho^2+\sigma^2)+\delta\rho\sigma\right)\,dx\notag\\
&\quad
=\int_\Omega \left(\fr{D}{2}-\fr{\delta^2}{2D}\right)\rho^2+\fr{\delta^2}{2D}(\rho-\bar\rho)^2+\fr{D}{2}(\sigma-\bar\sigma)^2+\delta\rho(\sigma-\bar\sigma)+\fr{D}{2}\left(\bar\sigma+\fr{\delta}{D}\bar\rho\right)^2\,dx.
\la{mess}
\end{align}
By \eqref{deltaD}, the first term on the right hand side of \eqref{mess} is positive. Next, we make two observations. First,
\be
\bal
Q_1=&\int_\Omega\fr{\delta^2}{2D}(\rho-\bar\rho)^2+\fr{D}{2}(\sigma-\bar\sigma)^2+\delta\rho(\sigma-\bar\sigma)\,dx\\
=&\int_\Omega\fr{\delta^2}{2D}(\rho-\bar\rho)^2+\fr{D}{2}(\sigma-\bar\sigma)^2+\delta(\rho-\bar\rho)(\sigma-\bar\sigma)\,dx\\
=&\fr{1}{2}\int_\Omega\left(\fr{\delta}{D^\fr{1}{2}}(\rho-\bar\rho)+D^\fr{1}{2}(\sigma-\bar\sigma)\right)^2\,dx \ge 0.
\eal
\ee
Second, we observe that by integrating (\ref{dD}) and using the boundary conditions (\textbf{EN}), we reach the conclusion that the quantity $\delta\bar\rho+D\bar\sigma$ is independent of time. Therefore, taking the time derivative of (\ref{mess}), we arrive at
\be
\fr{d}{dt}\int_\Omega \left(\fr{D}{2}(\rho^2+\sigma^2)+\delta\rho\sigma\right)\,dx
=\fr{d}{dt}\left(\int_\Omega \left(\fr{D}{2}-\fr{\delta^2}{2D}\right)\rho^2\,dx +Q_1\right).
\ee
Now, we denote
\begin{align}
Q&=\fr{1}{D_1D_2}\left(\int_\Omega \left(\fr{D}{2}-\fr{\delta^2}{2D}\right)\rho^2\,dx+Q_1\right),\\
R&=\int_\Omega |\na\sigma|^2+|\na\rho|^2\,dx,
\end{align}
so that (\ref{l2diss'}) is equivalent to
\be
\fr{d}{dt}Q+R\le 0.
\la{l2diss}
\ee
Next we note that, by Poincar\'{e}'s inequality, we have the bound
\be
\bal
Q_1&\le \tilde{C}_\Omega\int_\Omega |\rho-\bar\rho|^2+|\sigma-\bar\sigma|^2\,dx\\
&\le C_\Omega R
\eal
\la{q1>r}
\ee
for constants $\tilde{C}_\Omega$ and $C_\Omega$ depending on $\Omega$ and $D_i$. We also have, again by the Poincar\'{e}  inequality,
\be
R\ge C'_\Omega\int_\Omega\rho^2\,dx
\la{r>rho2}
\ee 
for $C'_\Omega$ depending only on $\Omega$. Combining (\ref{q1>r}) and (\ref{r>rho2}), we obtain
\be
R\ge C_D Q
\la{r>q}
\ee
where $C_D$ is a constant depending on $\Omega$ and $D_i$. So using  (\ref{r>q}), the differential inequality (\ref{l2diss}) gives
\be
\fr{d}{dt}Q+C_DQ\le 0,
\ee
so that $Q(t)\le e^{-C_Dt}Q(0)$. Finally, defining
\be
P=\fr{1}{D_1D_2}\int_\Omega \left(\fr{D}{2}-\fr{\delta^2}{2D}\right)\rho^2\,dx
\ee
and recalling that $Q_1\ge 0$, we have
\be
\begin{aligned}
Q&\ge P
\la{q>p}
\end{aligned}
\ee
so that
\be
P(t)\le e^{-C_Dt}Q(0).
\ee
That is,
\be
\|\rho(t)\|_{L^2}^2\le Ce^{-C_Dt}
\ee
where
\be
C=D_1D_2\left(\fr{D}{2}-\fr{\delta^2}{2D}\right)^{-1}Q(0).
\ee
Therefore, the proof of Theorem~\ref{decay} is complete.
\end{proof}

We prove now a maximum principle.
\begin{thm}\la{uben}
For global smooth solutions $(c_1,c_2,\Phi,u)$ of the NPNS system with (\textbf{EN}) boundary conditions, the ionic concentrations $c_i$ obey the uniform in time bounds
\be
c_i(x,t)\le \max\left\{\sup_\Omega c_1(0), \sup_\Omega c_2(0)\right\},\quad i=1,2.
\ee
\end{thm}
\begin{proof}
The idea of the proof is as follows. We define
\be
\bal
m_i(t)&=\sup_\Omega c_i(x,t),\quad i=1,2,\\
m(t)&=\max\{m_1(t),m_2(t)\},\\
M(t)&=\max_{s\le t}m(s).
\eal
\ee
Then the statement of the theorem is equivalent to the statement that $M(t)=M(0)$ for all $t\ge 0$. For the sake of contradiction, if we assume that $M(t)$ in fact increases, then we have $M'(t)>0$ for some $t>0$. At time $t$, we deduce that (without loss of generality) $M(t)=m_1(t)\ge m_2(t)$ and that $\pa_t c_1(x,t)>0$ for all $x\in B_t^1=\{x\in\bar\Omega\,|\, c_1(x,t)=M(t)\}$. Then, an argument like that leading up to (\ref{<0}) allows us to deduce that in fact $B_t^1\subset\pa \Omega$. But then by Hopf's lemma, we conclude that $\pa_n c_1(x,t)>0$ for $x\in B_t\subset\pa\Omega$. However, then the boundary conditions force $\pa_n c_2(x,t)<0$, $c_2(x,t)=M(t)$. Consequently, it follows that at time $t$, $c_2$ attains an interior value exceeding $M(t)$, which is a contradiction.

In order to provide a rigorous  proof of Theorem~\ref{uben}, we note that in view of the fact that $c_i(x,t), i=1,2$ are smooth, we have that $m_i, m, M$ are Lipschitz in $t$ on any interval $[0,T]$. Indeed, there exist $x_i(t)\in\bar{\Omega}$ such that $m_i(t)=c_i(x_i(t),t)$. Then, for $t-s>0$ we have
\[\fr{c_i(x_i(s),t)-c_i(x_i(s),s)}{t-s}
\le\fr{m_i(t)-m_i(s)}{t-s}\le \fr{c_i(x_i(t),t)-c_i(x_i(t),s)}{t-s}, 
\]
and therefore
\[\left|\fr{m_i(t)-m_i(s)}{t-s}\right|
\le \sup_{\bar{\Omega}\times[0,T]} |\partial_t c_i|.
\]
In particular,  $m_i, m, M$ are differentiable a.e. and satisfy the fundamental theorem of calculus. This level of regularity is sufficient. 

We prove the following facts.

\noindent {\bf{(I) }} Let $m:[0,\infty)\to\mathbb{R}$ be locally Lipschitz (meaning Lipschitz on $[0,T]$ for any $T>0$), and let
\be
M(t)=\max_{s\le t}m(s).
\ee
Suppose for some $t>0$, $M'(t),m'(t)$ both exist and $M'(t)>0$. Then $M(t)=m(t)$ and $M'(t)=m'(t)$.

In order to check {\bf (I)}, we observe that $M'(t)>0$ implies that $M(s)<M(t)<M(r)$ for $s<t<r$. If not, then we had to have $M(s)=M(t)$ for some $s<t$. Then because it is nondecreasing, $M$ must be constant on the interval $[s,t]$, so that the left sided derivative of $M$ at $t$ is $0$, a contradiction. A similar argument is used for $t<r$. Next, we observe that $M'(t)>0$ implies that $m(t)=M(t)$. If not, then $M(t)=m(s)$ for some $s<t$, but then, from the previous observation, we obtain $M(s)<M(t)=m(s)\le M(s)$, which is a  contradiction. Lastly, we take $s<t$ and compute
\be
M(t)-M(s)=m(t)-M(s)\le m(t)-m(s)
\ee
so that dividing by $t-s$ and taking the limit $s\to t^-$, we obtain, $M'(t)\le m'(t)$. Similarly, by considering $s>t$, we obtain $M'(t)\ge m'(t)$.

\noindent{\bf{(II)}}  Let $c_i(x,t):\bar{\Omega}\times[0,\infty)\to\mathbb{R}$ be a smooth function with $\Omega\subset\mathbb{R}^d$ an open bounded set with smooth boundary. Let
\be
m_i(t)=\sup_{x\in\bar\Omega}c_i(x,t).
\ee
For each $t$, define $B_t^i=\{x\in\bar{\Omega}\,|\,c_i(x,t)=m_i(t)\}$. Suppose $m_i'(t)$ exists for some $t>0$. Then $\pa_tc_i(x,t)=m_i'(t)$ for each $x\in B_t^i$.

Indeed, we take $s<t$ and $x\in B_t^i$ and compute
\be
m_i(t)-m_i(s)=c_i(x,t)-m_i(s)\le c_i(x,t)-c_i(x,s)
\ee
from which we conclude that $m_i'(t)\le \pa_tc_i(x,t)$. A similar argument for $s>t$ gives the opposite inequality. An analogous argument gives the following fact.

\noindent{\bf (III)} Suppose $m_1(t)$ and $m_2(t)$ are locally Lipschitz, and let
\be
m(t)=\max\{m_1(t),m_2(t)\}.
\ee
Suppose $m'(t),m_1'(t),m_2'(t)$ exist for some $t>0$. Then for all $i\in\{1,2\}$ such that $m_i(t)=m(t)$, we have $m_i'(t)=m'(t)$.

To prove Theorem~\ref{uben}, we assume for the sake of contradiction that for some $T>0$ we have $M(T)>M(0)$. We define
\be
\bal
A_i&=\{t\in(0,T)\,|\, m_i'(t)\text{ exists}\},\quad i=1,2,\\
A_m&=\{t\in(0,T)\,|\, m'(t)\text{ exists}\},\\
A_M&=\{t\in(0,T)\,|\, M'(t)\text{ exists}\},\\
A&=A_1\cap A_2\cap A_m\cap A_M.
\eal
\ee
Since all the functions under consideration are locally Lipschitz, we know that $A$ has full measure, $|A|=T$. So, by the fundamental theorem of calculus, there exists some $t\in A\subset(0,T)$ such that $M'(t)>0$. By the considerations above, we conclude without loss of generality that  $\pa_t c_1(x,t)=M'(t)>0$ and $c_1(x,t)=m_1(t)=M(t)\ge m_2(t)$ for all $x\in B_t^1$.\\
\indent We claim that $B_t^1\subset\pa\Omega$. Indeed, if there were some $x\in B_t^1\cap \Omega$, then evaluating (\ref{np}) at $(x,t)$, and using $\na c_1(x,t)=0$ and $\D c_1(x,t)\le 0$, we obtain $\pa_tc_1(x,t)\le 0$, a contradiction. Now we fix $x\in B_t^1\subset\pa\Omega$. Then, at this boundary point (and at time $t$) the function $F=-\pa_tc_1-\fr{D_1}{\epsilon}\rho$ satisfies $F< 0$. We take a small open subset $U\subset\Omega$ that shares an open boundary portion with $\Omega$, including the point $x$. We choose the subset to be small enough i.e. uniformly close enough to $x$, so that $F_{|U}(\cdot,t)\le 0$. Then, restricted to $U$, we have that $c_1$ satisfies, at time $t$,
\be
-D_1\D c_1+(u-D_1\na\Phi)\cdot\na c_1=F\le 0.
\ee
Thus, by Hopf's lemma we conclude that $\pa_n c_1(x,t)>0$. Then, the boundary conditions imply that $\pa_n c_2(x,t)<0$ and $c_2(x,t)=M(t)$. That is, $c_2$ attains an interior value strictly greater than $M(t)$ at time $t$, a contradiction. This completes the proof.

\end{proof}

We prove now exponential decay in $L^p$ for all $p<\infty$.

\begin{cor}
For global smooth solutions $(c_1,c_2,\Phi,u)$ of the NPNS system with (\textbf{EN}) boundary conditions, there exist constants $\lambda_p>0$ depending on $\Omega$, $D_i$ and $p<\infty$, and $C_p>0$ depending additionally on initial conditions, such that
\be
\|\rho(t)\|_{L^p}\le C_pe^{-\lambda_p t},\quad t\ge 0
\ee
holds.
\end{cor}
\begin{proof} The proof follows by interpolation from Theorem~\ref{decay} and
Theorem~\ref{uben}
\end{proof}
\end{section}

\vspace{.5cm}

{\bf{Acknowledgment.}} The work of PC was partially supported by NSF grant DMS-
171398.

\end{document}